\documentclass{amsart}
\usepackage[utf8]{inputenc}

\usepackage[foot]{amsaddr}

\usepackage{biblatex}
\usepackage{mathtools}
\usepackage{amsmath}
\usepackage{amssymb}
\usepackage{amsthm}
\usepackage{stmaryrd}
\usepackage{hyperref}
\usepackage{cleveref}
\usepackage{listings}

\usepackage[hyphenbreaks]{breakurl}

\usepackage{xcolor}
\usepackage{dutchcal}
\usepackage{tikz-cd}
\usetikzlibrary{calc,decorations.markings,decorations.pathmorphing}
\usepackage{ifthen}
\usepackage{bm}

\theoremstyle{definition}
\newtheorem{Def}[]{Definition}[section]
\newtheorem{Pro}[Def]{Proposition}
\newtheorem{Rem}[Def]{Remark}
\newtheorem{Lem}[Def]{Lemma}
\newtheorem{Exp}[Def]{Example}
\newtheorem{Thm}[Def]{Theorem}
\newtheorem{Cor}[Def]{Corollary}
\newtheorem{Ques}[Def]{Question}

\newtheorem*{Ques*}{Question}
\newtheorem*{Thm*}{Theorem}
\newtheorem*{Cor*}{Corollary}

\DeclareMathOperator{\RHom}{\mathbb{R}\strut\kern-.2em\operatorname{Hom}}
\DeclareMathOperator{\Aut}{Aut}

\DeclareMathOperator{\Ker}{Ker}

\DeclareMathOperator{\spec}{Spec}

\DeclareMathOperator{\proj}{proj}
\DeclareMathOperator{\inc}{inc}

\newcommand{\bbone}{\text{\usefont{U}{bbold}{m}{n}1}}
\MakeRobust{\bbone}

\renewcommand{\Im}{\operatorname{Im}}

\title{Isomorphisms of graded semiconnected algebras}

\author[Darius Dramburg]{Darius Dramburg}
\address{Darius Dramburg, Kavli Institute for the Physics and Mathematics of the Universe (WPI),The University of Tokyo Institutes for Advanced Study, The University of Tokyo, Kashiwa, Chiba 277-8583, Japan}

\email{darius.dramburg@ipmu.jp}


\date{\today}

\begin{document}

\begin{abstract}
    Let $A_\bullet$ and $B_\bullet$ be nonnegatively graded algebras, finitely generated in degrees less than $2$ and with semisimple base rings $A_0$ and $B_0$. We prove that if $A \simeq B$ as ungraded algebras, then $A_\bullet \simeq B_\bullet$ as graded algebras. 
    This generalises a theorem of Bell and Zhang \cite{BellZhang} in the connected case $A_0 = k = B_0$, and the path-algebra version when $A_0$ and $B_0$ are $k$-elementary due to Gaddis \cite{Gaddis}. We also discuss some related open problems for graded algebras. 
\end{abstract}

\maketitle

\section{Introduction}
In \cite{BellZhang}, Bell and Zhang showed that for connected graded algebras $A_\bullet$ and $B_\bullet$, finitely generated in degrees $0$ and $1$, the existence of an (ungraded) isomorphism $A \simeq B$ implies the existence of a graded isomorphism $A_\bullet \simeq B_\bullet$. This result was generalised by Gaddis to the case of graded path algebras with homogeneous relations \cite{Gaddis}.
The purpose of this article is to generalise this result further by removing the assumption that the base rings $A_0$ and $B_0$ are $k$-elementary.

Throughout, we work over a fixed field $k$, and all gradings are over the integers. We will use the following terminology. 
\begin{Def}
    A graded algebra $A_\bullet = \bigoplus_{i \in \mathbb{Z}} A_i$ is called 
    \begin{enumerate}
        \item \emph{semiconnected} if $A$ is nonnegatively graded, i.e. $ A_\bullet = \bigoplus_{i = 0}^\infty A_i$, and $A_0$ is semisimple. 
        \item \emph{standard graded} if $A$ is generated by $A_0 \cup A_1$ as an algebra. 
        \item \emph{locally finite} if $\dim_k(A_i)$ is finite for all $i \in \mathbb{Z}$.  
    \end{enumerate}
\end{Def}

In \cite{GrantIyama}, a grading satisfying all three properties is called a \emph{radical grading}. Our main result is the following.

\begin{Thm}\label{Thm: Main thm}
    Let $A_\bullet$ and $B_\bullet$ be semiconnected standard graded locally finite algebras. If $A \simeq B$ as ungraded algebras, then $A_\bullet \simeq B_\bullet$ as graded algebras.
\end{Thm}

In particular, this shows that while a given algebra may admit many different locally finite semiconnected standard gradings, all of them are graded isomorphic.

\begin{Cor}
    If $A$ admits a locally finite semiconnected standard grading, then the grading is unique up to graded automorphism. 
\end{Cor}

The above corollary is interesting especially in cases where an algebra is known to have a certain property, such as being Koszul or more generally being quadratic, but this property requires the choice of a grading. 

Our result recovers some existing results in the literature. When $A_0 = k = B_0$ is just the base field, \Cref{Thm: Main thm} recovers the result of Bell-Zhang for locally finite standard graded connected algebras. The generalisation of Gaddis to graded path algebras is recovered when $A_0$ and $B_0$ are not just semisimple but $k$-elementary, that is $A_0 \simeq k \oplus \cdots \oplus k$. A quiver for $A$ can be constructed by taking vertices as idempotents from $A_0$ and arrows as well-chosen generators of $A_1$ as an $A_0$-bimodule. The relations are then homogeneously generated in degree at least $2$.

\subsection{Sketch of proof}
Our proof follows the ideas laid out by Bell-Zhang \cite{BellZhang} and by Gaddis \cite{Gaddis}. In particular, we follow the strategy that one should study ideals of a prescribed codimension and tangent dimension. However, we avoid choosing a basis for the degree $1$ part of our algebras, i.e. the tangent bimodule, and thus circumvent the use of commutators entirely. To sketch the proof, let $A_\bullet$ and $B_\bullet$ be locally finite semiconnected standard graded, with their positive degree ideals denoted $A_+$ and $B_+$. Let $\varphi \colon A \to B$ be an isomorphism, and denote $I = \varphi^{-1}(B_+)$. We first compare the base rings and show that $A_0 \simeq B_0$ via the degree $0$ component of $\varphi$. Then, we construct an isomorphism $A_1 \simeq B_1$ that is compatible with $\varphi_0$. Since $A$ is generated by $A_0 \cup A_1$, it then suffices to extend this multiplicatively. Doing so requires factoring an element $a \in A_n$ of degree $n$ into factors of degree $1$. We do this formally by passing through the tensor algebra $\operatorname{T}_{A_0} A_1$ that surjects onto $A$, so the construction of the candidate for a graded isomorphism already contains the check that it respects relations. This check relies on an induction. To show that everything is well-defined, we need to compare the filtration of $A$ by powers of $A_+$ with that by powers of $I$.

\section{Preliminaries}
We collect some easy observations on graded rings and set up notation. For this, we let $A = \bigoplus_{i \in \mathbb{Z}} A_i $ be a graded ring, and $A_+ = \bigoplus_{i > 0} A_i$. We refer the reader to \cite{MethodsOfGradedRings} for background on graded rings and algebras.  

\begin{Rem}
    Let $A_\bullet$ be a standard graded algebra. This implies that the grading is nonnegative. For all $n \geq 0$, we have the equality of ideals
    \[ A_+^n = A_{\geq n} = \bigoplus_{i \geq n} A_i. \]
    In particular, writing $A_+^0 = A$, we have the following isomorphism of graded algebras 
    \[ A_\bullet \simeq \operatorname{gr} A_\bullet = \bigoplus_{i \geq 0} A_+^i/A_+^{i+1}.  \]
\end{Rem}

We need the following notation throughout. 

\begin{Def}
    Let $A_\bullet$ be a graded algebra. Then we denote by $\proj_i^A \colon A \to A_i$ the projection onto the homogeneous component of degree $i$. We drop the superscript $A$ when no confusion is possible. 
\end{Def}

Before we continue, we point out an alternative interpretation of our assumptions. 

\begin{Rem}
    Let $A_\bullet$ be a graded algebra. Then the grading is standard and semiconnected if and only if $A_0$ is semisimple and there exists an $A_0$-bimodule $M$ such that $A_\bullet$ is graded isomorphic to a quotient of the tensor algebra $T_{A_0} M$ by an ideal that is homogeneous with respect to tensor degrees and contained in degrees at least $2$.  
\end{Rem}

We also briefly comment on the to our knowledge uncommon usage of the word ``semiconnected''. The word has the pleasant feature of being an obvious generalisation of ``connected'' while also borrowing the ``semi'' from the defining ``semisimple'' ring $A_0$. Note that a semiconnected, locally finite, standard grading is called a \emph{radical grading} in \cite{GrantIyama}.  

\begin{Rem}
    A graded $k$-algebra $A_\bullet$ is called connected if the grading is nonnegative and $A_0 = k$. In the commutative case, connectedness is equivalent to requiring that $\operatorname{Proj}(A)$ is a connected space. For a semiconnected commutative algebra, the number of orthogonal primitive central idempotents in $A_0$ determines the number of connected components of $\operatorname{Proj}(A)$. Note that passing from $A_0 = k$ to $A_0$ being semisimple also brings with it non-reducedness because $A_0$ need not be basic. In this way, the case covered by Gaddis in \cite{Gaddis} can be seen as the noncommutative \emph{reduced} semiconnected case. 
\end{Rem}

\section{The base ring}
From now on, we assume that $A_\bullet$ and $B_\bullet$ are locally finite standard graded semiconnected algebras. We denote by $A_+ = \bigoplus_{i > 0} A_i$ and $B_+ = \bigoplus_{i > 0} B_i$ their positive degree ideals, and by $\varphi \colon A \to B$ an ungraded isomorphism.  

The following observation about idempotents holds for arbitrary graded rings. 

\begin{Lem}\label{Lem: Positive idempotents are zero}
    Let $e \in A$ be an idempotent. If $e \in A_+$, then $e = 0$. 
\end{Lem}

\begin{proof}
    By assumption, we have $e \in A_+$ and hence $e^n \in A_{\geq n}$ for all $n \geq 1$. However, since every element in a graded algebra has finite support, it follows that $e^m$ and $e$ have disjoint support for large enough $m$. Since $e^m = e $, our claim follows. 
\end{proof}

\begin{Def}
    For the isomorphism $\varphi \colon A \to B$, we denote the degree $0$ component as 
    \[\varphi_0 \colon A_0 \to B_0, a \mapsto \proj_0^B(\varphi(a)). \]
\end{Def}

\begin{Pro}\label{Pro: phi_0 is iso}
    The map $\varphi_0$ is an isomorphism of algebras.
\end{Pro}

\begin{proof}
    Since $B_0 \simeq B/B_+$, the map $\proj_0^B$ is an algebra homomorphism, and its kernel is $\Ker(\proj_0^B) = B_+$. 
    Therefore, $\varphi_0 = \proj_0^B \circ \; \varphi$ is an algebra homomorphism, which means that $\Ker(\varphi_0) \subseteq A_0$ is an ideal. We show that $\Ker(\varphi_0) = 0$. By the Artin-Wedderburn theorem, the semisimple algebra $A_0$ decomposes as $ A_0 \simeq \Pi_{i = 1}^n D_i^{m_i \times m_i}$ for division algebras $D_i$ over $k$. Since each matrix algebra $D_i^{m_i \times m_i}$ has only trivial ideals, a non-trivial ideal in $A_0$ needs to contain an ideal isomorphic to one of the $D_i^{m_i \times m_i}$, which means that a non-trivial ideal in $A_0$ contains a non-trivial idempotent. However, because $\varphi$ is injective, we have that every idempotent $e \in \Ker(\varphi_0)$ must satisfy $\varphi(e) \in B_+$. By \Cref{Lem: Positive idempotents are zero}, it follows that $\varphi(e) = 0$, and hence $e=0$. Thus we conclude that $\Ker(\varphi_0)= 0$. 
    We then reverse the roles of $A$ and $B$ via $\varphi^{-1}$ to see that $B_0$ also injects into $A_0$. Since $A_0$ and $B_0$ are finite dimensional, it follows that $\dim(A_0) = \dim(B_0)$, and hence $\varphi_0$ is an isomorphism.  
\end{proof}

\section{Split ideals and their tangent bimodules}\label{Sec: Split and tangent}
With the same assumptions on $A$ and $B$, we now consider $I = \varphi^{-1}(B_+) \trianglelefteq A$. The goal is to show that $I$ and $A_+$ behave similarly. The first steps are all internal to $A$, so we consider an arbitrary ideal $J \trianglelefteq A$ and we write $\proj_{A/J} \colon A \to A/J$ for the canonical projection and $\inc \colon A_0 \to A$ for the inclusion morphism. 

\begin{Def}
    A (possibly ungraded) ideal $J \trianglelefteq A$ is \emph{split} (with respect to the fixed grading $A_\bullet$) if the composition 
    \[ A_0 \xrightarrow[]{\inc} A \xrightarrow{\proj_{A/J}} A/J \]
    is an isomorphism. In this case, we denote the retraction $A \to A_0$ by  
    \[ \pi_J = (\proj_{A/J} \circ \inc)^{-1} \circ \proj_{A/J}. \]
\end{Def}

In the terms of \cite{BellZhang, Gaddis}, a split ideal $J$ is an ideal of codimension $\operatorname{dim}(A_0)$. 

\begin{Rem}
    Calling $\pi_J$ a retraction is justified: We have that $\pi_J$ restricts to the identity on $A_0$ and $\Ker(\pi_J) = J$. In particular, for a split ideal $J$ we have that $A = A_0 \oplus J$. We chose the above definition in favor of the equivalent characterization via $\pi_J$ because it makes the following more clear. The ideal $A_+$ gives an obvious identity composition 
    \[ A_0 \xrightarrow[]{\inc} A \xrightarrow{\proj_{A/A_+}} A/A_+, \]
    and we are simply substituting $A_+$ with any ideal $J$ for which this stays an isomorphism. 
\end{Rem}

\begin{Def}
    For a split ideal $J \trianglelefteq A$, we denote by $\pi_{J,1} = \pi_J|_{A_1}$ the restriction  of $\pi_J$ to $A_1$, and we fix $\psi_J = \operatorname{id}_{A_1} - \; \pi_{J,1} $ as a map $A_1 \to A$. 
\end{Def}

The following observations only require linear algebra. 

\begin{Pro}\label{Pro: Split ideal properties}
    For any split ideal $J \trianglelefteq A$ we have the following: 
    \begin{enumerate}
        \item The map $\pi_{J,1}$ is an $A_0$-bimodule morphism. 
        \item The map $\psi_J$ is an injective $A_0$-bimodule morphism. 
        \item The image of $\psi_J$ satisfies $\Im(\psi_J) \subseteq J$. 
    \end{enumerate}
\end{Pro}

\begin{proof}\leavevmode
    \begin{enumerate}
        \item Recall that $\pi_J$ is a ring homomorphism which restricts to the identity on $A_0$. Therefore, we have for $a, a' \in A_0$ and $m \in A_1$ that 
        \begin{align*}
            \pi_{J,1} ( a m a') &= \pi_J(ama') = \pi_J(a) \pi_J(m) \pi_J(a')\\ &= a \pi_J(m) a' = a \pi_{J,1}(m) a'.
        \end{align*} 
        \item The map $\psi_J$ is a bimodule morphism because it is the difference of two bimodule morphisms. To see that it is injective, suppose $\psi_J(m) = 0$ for $m \in A_1$. Then we have 
        \[ 0 = \psi_J(m) = m - \pi_{J}(m), \]
        but $m \in A_1$ and $\pi_J(m) \in A_0$, and hence $m=0$.
        \item To see that $\Im(\psi_J) \subseteq J$, recall that $J = \Ker(\pi_J)$ and apply $\pi_J$ to $\Im(\psi_J)$. For $m \in A_1$, we have
        \[ \pi_J(\psi_J(m)) = \pi_J(m - \pi_J(m)) = \pi_J(m) - \pi_J(m) = 0. \qedhere \]
    \end{enumerate}
\end{proof}

The next property of $\psi_J$ is less obvious. 

\begin{Pro}\label{Pro: psi is epi onto J/J^2}
    With the same assumptions, we have that 
    \[ \overline{\psi_J} \colon A_1 \to J/J^2, m \mapsto \psi_J(m) + J^2 \]
    is an epimorphism of $A_0$-bimodules. 
\end{Pro}

\begin{proof} Our proof consists of three steps. 

\emph{Step 1.} We show that the ideal generated by the image of $\psi_J$ equals $J$: Let $J' = A\psi(A_1)A$. By the previous proposition, we know that $J' \subseteq J$. We have that $A_0 + J' = \{ a + j \mid a \in A_0, j \in J'\}$ is a subalgebra of $A$ because $J'$ is a two-sided ideal. Next, note that because every $m \in A_1$ can be written as 
    \[  m = \pi_J(m) + \psi_J(m) \in A_0 + J' \]
    the sum $A_0 + J'$ contains $A_0 \cup A_1$, and hence $A_0 + J' = A$. Furthermore, this is a direct sum, because $A_0 \cap J' \subseteq A_0 \cap J = 0$. 
    Thus, we have now found two complements of $A_0$ in $A$, i.e. we have 
    \[ A = A_0 \oplus J = A_0 \oplus J'. \]
    Together with the inclusion $J' \subseteq J$, it follows that $J = J'$.  

    \emph{Step 2.} We now know that $J$ is spanned by elements of the form $a \psi_J(m) a'$ for $m \in A_1$ and $a, a' \in A$. We claim that $a \psi_J(m) a'$ is congruent to $\pi_J(a) \psi_J(m) \pi_J(a')$ modulo $J^2$: 
    To see this, first note that taking the $J$-part of $a$ and $a'$ gives us elements $b = a - \pi_J(a)$ and $b' = a' - \pi_J(a')$ which lie in $J$. These terms appear when computing the difference
    \begin{align*}
        a \psi_J(m) a' - \pi_J(a) \psi_J(m) \pi_J(a') &= \pi_J(a) \psi_J(m) b' + b \psi_J(m) \pi_J(a') + b \psi_J(m) b',  
    \end{align*}
    and because they lie in $J$ we see that this difference is contained in $J^2 + J^2 + J^3 \subseteq J^2$. 
    Thus, we have that 
    \[ a \psi_J(m) a' + J^2 =  \pi_J(a) \psi_J(m) \pi_J(a') + J^2 \]
    as claimed. 
    
    \emph{Step 3.} We show that $\overline{\psi_J}$ is surjective: By the first step, we know that $J = J'$, so $J$ is generated by elements of the form $a \psi(m) a'$ for $a, a' \in A$ and $m \in A_1$. By the second step, we know that when passing to $J/J^2$, we may replace each $a \psi(m) a'$ with $\pi_J(a) \psi(m) \pi_J(a')$. It remains to note that this term is nothing but 
    \[ \pi_J(a) \psi_J(m) \pi_J(a') = \psi_J(\pi_J(a) m \pi_J(a')), \]
    and $\pi_J(a) m \pi_J(a') \in A_1$ because $\pi_J(a), \pi_J(a') \in A_0$. Thus, 
    \[ \overline{\psi_J}(\pi_J(a) m \pi_J(a')) + J^2 = a \psi_J(m) a' + J^2,   \]
    and we conclude that $\overline{\psi_J}$ is surjective. \qedhere
\end{proof}

\begin{Rem}
\begin{enumerate}
    \item In the terminology of \cite{BellZhang, Gaddis}, the tangent dimension of $J$ is $\dim(J/J^2)$, or a suitable refinement of this dimension. The above therefore states that $J$ has at most tangent dimension $\dim(A_1)$.
    \item So far, all computations happened inside $A$, and we do not know whether $\overline{\psi_J}$ has to be injective as well, i.e. we do not have equality of tangent dimensions for $A_+$ and $J$. To achieve this, we have to consider the case $J = I = \varphi^{-1}(B_+)$, and use that $\varphi$ is an isomorphism to play the argument backwards to prove injectivity of $\overline{\psi_I}$.
\end{enumerate}     
\end{Rem} 

From now on, we consider $J = I = \varphi^{-1}(B_+)$. We also need the notation $K = \varphi(A_+) \trianglelefteq B$. 

\begin{Pro}\label{Pro: I and K are split}
    Let $A_\bullet$ and $B_\bullet$ be semiconnected standard graded locally finite algebras and $\varphi \colon A \to B$ an ungraded isomorphism. Then $I = \varphi^{-1}(B_+)$ and $K =\varphi(A_+)$ are split in $A$ respectively $B$. 
    Furthermore, the map $\overline{\psi_I}$ is an $A_0$-bimodule isomorphism $A_1 \to I/I^2$, and $\overline{\psi_K}$ is a $B_0$-bimodule isomorphism $B_1 \to K/K^2$.  
\end{Pro}

\begin{proof}
    We only show the statement for $A$ and $\overline{\psi_I}$, the proof for $B$ is the same. 
    The fact that $A_0 \to A \to A/I$ is injective follows as before: If the kernel of $A_0 \to A \to A/I$ is non-trivial, it contains a non-trivial idempotent $e$, but then $\varphi(e) \in B_+$ and hence $e = 0$. To see surjectivity, take an arbitrary $a + I \in A/I$, and take $b \in A_0$ such that $\varphi_0(b) = \proj_0^B(\varphi(a))$, which exists because $\varphi_0$ is an isomorphism $A_0 \to B_0$ by \Cref{Pro: phi_0 is iso}. Then $\varphi(b -a) \in B_+$, and hence $a+I = b+I$. Thus the composition $A_0 \to A \to A/I$ is bijective. 
    
    To see that $\overline{\psi_I}$ is an isomorphism, we compare $\dim(A_1)$ and $\dim(I/I^2)$. We already know that $\overline{\psi_I}$ is surjective, so it suffices to show that $\dim(A_1) = \dim(I/I^2)$. Using that $\varphi$ is an isomorphism, we see that $\dim(I/I^2) = \dim(B_+/B_+^2) = \dim(B_1)$ and $\dim(K/K^2) = \dim(A_+/A_+^2) = \dim(A_1)$. Using the surjections $\overline{\psi_I}$ and $\overline{\psi_K}$, we have 
    \[ \dim(A_1) \geq \dim(B_+/B_+^2) = \dim(B_1) \geq \dim(A_+/A_+^2) = \dim(A_1). \qedhere \] 
\end{proof}

\section{Constructing the isomorphism}
We now have all necessary components to build our graded isomorphism. We use $\varphi_0$ in degree $0$ and we use $ A_1 \xrightarrow{ \overline{\psi_I}} I/I^2 \to B_1$ in degree $1$. Since $A$ is generated by $A_0 \cup A_1$ we would like to extend this multiplicatively in a direct fashion. However, to write down this multiplicative extension in degree $n$, we need to factor degree $n$ elements $a \in A_n$ into $n$ factors of degree $1$. We do this by passing through the tensor algebra $\operatorname{T}_{A_0}(A_1)$. As before, we first work only with $A$. 

\begin{Rem}
    For $n \geq 1$, we fix notation for the $n$-fold tensor product 
    \[A_1^{(n)} = A_1 \otimes_{A_0} \cdots \otimes_{A_0} A_1. \]
    We then have that the multiplication map $\mu_n \colon A_1^{(n)} \to A_n$ is surjective because $A$ is generated by $A_0 \cup A_1$. 
\end{Rem}

\begin{Def}
    With the standing assumptions on $\varphi \colon A \to B$, we fix the notation 
    \[ \psi^{(n)} \colon A_1^{(n)} \to I^n, (a_1 \otimes \cdots \otimes a_n) \mapsto \psi_I(a_1) \cdots \psi_I(a_n). \]
    We furthermore define $ \psi^{(0)} \colon A_0 \to A, a \mapsto a$.
\end{Def}

We first show that the highest degree part of $\psi^{(n)}$ is what we expect. We use the notation $A_{< n} = \bigoplus_{0 \leq i < n} A_i$ and $A_{\leq n} = \bigoplus_{0 \leq i \leq n} A_i$ in the following. Note that these are only $A_0$-subbimodules of $A$. 

\begin{Lem}\label{Lem: psi product top degree}
    With the above setup, we have for $n \geq 1$ and $a_1, \ldots, a_n \in A_1$ that 
    \[ \psi_I(a_1) \cdots \psi_I(a_n) - a_1 \cdots a_n \in A_{< n}. \]
\end{Lem}

\begin{proof}
    Recall that $\psi_I(a) = a - \pi_I(a)$ with $\pi_I(a) \in A_0$ and $a \in A_1$. Therefore, expanding the product $\prod_{i=1}^n \psi_I(a_i)$ gives the highest degree term $\prod_{i=1}^n a_i$ of degree $n$. All other terms contain at least one factor $\pi_I(a_i)$, and hence have degree strictly less than $n$. 
\end{proof}

Next, we show that the filtration by powers of $I$ behaves as expected. 

\begin{Lem}\label{Lem: Filtration comparison}
    For every $n \geq 1$, the following hold:
    \begin{enumerate}
        \item As ideals, we have $I^n = \psi_I(A_1)^{n} + I^{n+1}$.
        \item As an algebra, we have $A = A_{< n} + I^{n}.$
    \end{enumerate}
\end{Lem}

\begin{proof}
    We prove both statements by induction. 
    \begin{enumerate}
        \item When $n=1$, the fact that $I = \psi(A_1) + I^2$ follows from \Cref{Pro: psi is epi onto J/J^2}. For the induction, assume now that $I^{n-1} = \psi_I(A_1)^{n-1} + I^{n}$. Then we have 
    \begin{align*}
        I I^{n-1} &= I(\psi_I(A_1)^{n-1} + I^n) \\
        &= (\psi_I(A_1) + I^2)(\psi_I(A_1)^{n-1} + I^n) \\
        &= \psi_I(A_1)\psi_I(A_1)^{n-1} + I^2\psi_I(A_1)^{n-1} + \psi_I(A_1)I^n + I^2 I^n \\
        &= \psi_I(A_1)^{n} + I^2\psi_I(A_1)^{n-1} + \psi_I(A_1)I^n + I^{n+2},
    \end{align*}
    where the last three summands are all contained in $I^{n+1}$ because $\psi_I(A_1) \subseteq I$. Thus we have $I^n \subseteq \psi_I(A_1)^n + I^{n+1}$, and the reverse inclusion is obvious.
    
    \item When $n=1$, the fact that $A = A_0 + I$ follows from \Cref{Pro: I and K are split}. 
    For the induction, assume now that $A = A_{< n-1} + I^{n-1}$. Note that by \Cref{Lem: psi product top degree} we have 
    \[ \psi_I(A_1)^{n-1} \subseteq A_{n-1} + A_{<n-1} = A_{\leq n-1} = A_{< n}. \]
    This, together with the first statement for $I^{n-1}$, shows that
    \begin{align*}
        A = A_{< n-1} + I^{n-1} & = A_{< n-1} + \psi_I(A_1)^{n-1} + I^{n} \\ 
        &\subseteq  A_{< n-1} + A_{< n} + I^{n} = A_{<n} + I^{n+1}.
    \end{align*}  
    The reverse inclusion is obvious. \qedhere
    \end{enumerate}
\end{proof}

We now show that each of the above sums is direct, and produce the homogeneous components of the graded isomorphism. Recall that $I = \varphi^{-1}(B_+)$ and $K = \varphi(A_+)$ by definition, and recall our notation $\mu_n \colon A^{(n)} \to A_n$ for the multiplication map.

\begin{Pro}\label{Pro: Graded pieces}
    With the same assumptions, we have for every $n \geq 1$ the following.
    \begin{enumerate}
        \item The intersection $A_{<n} \cap I^n = 0$ and $B_{<n} \cap K^n = 0$.
        \item For $a \in A_n$, denote by $\hat{a} \in \mu_n^{-1}(a)$ an arbitrary preimage. Then the assignment 
        \[\hat{\psi}^{(n)} \colon A_n \to I^n, \; a \mapsto \psi^{(n)}(\hat{a})  \]
        gives a well-defined linear map. 
        This induces a well-defined linear bijection 
        \[ \psi_n \colon A_n \to B_n, \; a \mapsto \proj_n^B(\varphi(\hat{\psi}^{(n)}(a))).  \]
    \end{enumerate}
\end{Pro}

\begin{proof}
    We proceed by induction. In each step, we need the bijectivity of $\psi_j$ for all $j <n$ to establish the statement for $A_{<n} \cap I^n$, and then use this statement to prove the statement for $\psi_n$. 
    
    \emph{Base case:} When $n = 1$, statement (1) follows from \Cref{Pro: I and K are split}. For statement (2), recall that $\overline{\psi_I} \colon A_1 \to I/I^2$ is bijective by \Cref{Pro: I and K are split}. The map $\psi_1$ is simply the composition of $\overline{\psi_I}$ with the isomorphism $I/I^2 \to B_1$ induced by $\varphi$, because the intermediate $\mu_1$ is the identity.

    \emph{Induction:} Assume that both statements hold for all $1 \leq j < n$. In particular, it follows from the bijectivity of $\psi_j$ and \Cref{Pro: phi_0 is iso} that $\dim(A_j) = \dim(B_j)$ for $0 \leq j < n$. For statement (1), note that by \Cref{Lem: Filtration comparison}, we have that $A_{<n} \to A/I^n$ is surjective. Now it suffices to compare dimensions along the filtration by powers of $I$. Write $I^0 = A$, then we have 
    \begin{align*}
        \dim(A/I^n) = \sum_{j = 0}^{n-1} \dim(I^j/I^{j+1}) = \sum_{j = 0}^{n-1} \dim(B_j) = \sum_{j = 0}^{n-1} \dim(A_j) = \dim(A_{<n}).  
    \end{align*}
    Thus, the surjection $A_{<n} \to A/I^n$ has to be a bijection, and hence $A_{<n} \cap I^n = 0$. The corresponding statement for $B$ is proven in the same way. 
    
    For statement (2), we first show that $\hat{\psi}^{(n)}$ is well-defined. The first step is to note that $\mu_n \colon A^{(n)} \to A_n$ is surjective, and hence taking a preimage $\hat{a}$ for every $a \in A_n$ is well-defined. Next, we show that the choice of preimage does not matter. For this, let $r \in \Ker(\mu_n)$, and write it as a sum $r = \sum_{i} a_{i,1} \otimes \cdots \otimes a_{i,n}$. 
    Applying $\psi^{(n)}$ gives 
    \begin{align*}
        \psi^{(n)}(r) = \sum_i \psi_I(a_{i,1}) \cdots \psi_I(a_{i,n}) = \sum_i a_{i,1} \cdots a_{i,n}  + b_i 
    \end{align*}
    where the $b_i$ are in $A_{<n}$ by \Cref{Lem: psi product top degree}. Since $r \in \Ker(\mu_n)$, it follows that $\sum_i a_{i,1} \cdots a_{i,n} = 0$, and hence $\psi^{(n)}(r) \in A_{<n}$. However, by \Cref{Pro: Split ideal properties} applied to $\psi_I$, we get that $\psi^{(n)}(r) =\sum_i \psi_I(a_{i,1}) \cdots \psi_I(a_{i,n})  \in I^n$. By statement (1), we have that $A_{<n} \cap I^n = 0$, and hence $\psi^{(n)} (r) = 0$. Thus, $\hat{\psi}^{(n)}$ is well-defined. To check that the resulting $\psi_n = \proj_n^B \circ \; \varphi \circ \hat{\psi}^{(n)}$ is bijective, note that $\Im(\hat{\psi}^{(n)}) = \Im(\psi^{(n)}) = \psi(A_1)^n $. By \Cref{Lem: Filtration comparison}, we have that $I^n = \psi_I(A_1)^n + I^{n+1} = \Im(\hat{\psi}^{(n)}) + I^{n+1}$, and therefore the composition 
    \[ A_n \xrightarrow{\hat{\psi}^{(n)}} I^n \to I^n/I^{n+1} \to B_+^n/B_+^{n+1} \simeq B_n  \]
    is surjective. But this composition is simply $\psi_n$, and playing the argument backwards shows that $\psi'_n \colon B_n \to A_n$, constructed from $\psi_K$, is also surjective. Comparing dimensions then shows that $\psi_n$ and $\psi_n'$ are bijections. \qedhere   
\end{proof}

It remains to check that the collection $(\varphi_0, \psi_1, \ldots)$ indeed constitutes an isomorphism of graded algebras, not just of graded vector spaces. To complete the notation, we write $\hat{\psi}^{(0)} = \operatorname{id}_{A_0}$ and $\psi_0 = \varphi_0$, as well as $A_1^{(0)} = A_0$ and $\mu_0 = \operatorname{id}_{A_0}$. 

\begin{Pro}\label{Pro: Multiplicativity}
    With the same assumptions, the following holds for all $n, n' \geq 0$ and $a \in A_n$, $a' \in A_{n'}$.  
    \begin{enumerate}
        \item In $A$, we have $\hat{\psi}^{(n)}(a) \hat{\psi}^{(n')}(a') = \hat{\psi}^{(n+n')}(aa') $.
        \item In $B$, we have that $\psi_n(a) \psi_{n'}(a') = \psi_{n+n'}(aa') $.
    \end{enumerate}
\end{Pro}

\begin{proof}
    We fix preimages $\hat{a} \in \mu_n^{-1}(a)$ and $\hat{a}' \in \mu_{n'}^{-1}(a')$. 
    \begin{enumerate}
        \item With these preimages, we have 
    \[ \psi_I^{(n)}(\hat{a}) \psi_I^{(n')}(\hat{a}') = \psi_I^{(n+n')}(\hat{a} \otimes \hat{a}').  \]
    Since $\hat{a} \otimes \hat{a} \in \mu_{n+n'}^{-1}(aa')$, we find that 
    \[ \quad \hat{\psi}^{(n)}(a) \hat{\psi}^{(n')}(a') =  \psi_I^{(n)}(\hat{a}) \psi_I^{(n')}(\hat{a}') = \psi_I^{(n+n')}(\hat{a} \otimes \hat{a}') = \hat{\psi}^{(n+n')}(aa'). \]
    \item We keep the same notation as before. Then we need to apply $\varphi$ and project to degree $n$, i.e. we write 
    \[ \varphi(\hat{\psi}^{(n)}(a)) = \psi_n(a) + \tau_{ n}(a) \]
    for $\tau_{n}(a) \in B_{> n} = B_+^{n+1}$, and similarly for $a'$. Applying $\varphi$ to the identity from the first step and expanding, we find 
    \begin{align*}
        \varphi \left(\hat{\psi}^{(n)}(a) \hat{\psi}^{(n')}(a') \right) &= \varphi(\hat{\psi}^{(n)}(a)) \cdot  \varphi( \hat{\psi}^{(n')}(a')) \\
        &= (\psi_n(a) + \tau_{n}(a)) \cdot  (\psi_{n'}(a') + \tau_{ n'}(a')) \\
        &= \psi_n(a)\psi_{n'}(a') \\
        &+ \tau_{n}(a)\psi_{n'}(a') + \psi_n(a)\tau_{ n'}(a') + \tau_{n}(a) \tau_{n'}(a').
    \end{align*}
    Then it remains to note that the last three summands lie in $B_{\geq n+1}$ and hence vanish under $\proj_{n+n'}^B$. Thus, we have 
    \[ \psi_{n+n'}(aa') = \proj_{n+n'}\left(\varphi \left(\hat{\psi}^{(n)}(a) \cdot  \hat{\psi}^{(n')}(a') \right) \right) = \psi_n(a) \psi_{n'}(a'). \qedhere \]
    \end{enumerate}
\end{proof}

\begin{proof}[Proof of \Cref{Thm: Main thm}]
    By \Cref{Pro: Graded pieces}, the map $\Psi = \bigoplus_{i \geq 0} \psi_i$ is a graded $k$-linear bijection $A \to B$ with $\Psi(A_i) = B_i$. Furthermore, by \Cref{Pro: Multiplicativity}, it respects mutliplication, and since $\psi_0 = \varphi_0$ is an algebra homomorphism, $\Psi$ is an isomorphism of graded unital $k$-algebras. 
\end{proof}

We note the an immediate application. The following version of Zariski cancellation is a straighforward generalisation of \cite{BellZhang} and \cite{Gaddis}. 

\begin{Thm}
    Let $A_\bullet$ and $B_\bullet$ be semiconnected standard graded locally finite algebras. If $\operatorname{Z}(A) \cap A_1 = 0$ and 
    \[ A[t_1, \cdots, t_n ] \simeq B[t_1, \cdots, t_n] \]
    as ungraded algebras, then $A \simeq B$. 
\end{Thm}

\section{Further questions and observations}
We conclude by pointing out some natural questions arising from the above arguments. These questions range in difficulty and technicality, but are all centered around similar observations concerning two gradings on isomorphic algebras. 

The first natural question to ask is how exactly the investigation of $A_+$ and $I$ fits into the framework of \cite{BellZhang, Gaddis}. In the following, we simply refer to dimension over $k$ as in \cite{BellZhang}, but of course the statements may, and probably should, be refined to dimension vectors or similar as in \cite{Gaddis}. We used fundamentally that $I$ is an ideal of codimension $\dim(A_0)$ and tangent dimension $\dim(A_1)$. However, it is not clear that every codimension $\dim(A_0)$ ideal in $A$ is split. 

\begin{Ques}
    Let $A_\bullet$ be standard graded, locally finite and semiconnected. Let $J \trianglelefteq$ be an ideal with $\dim(A/J) = \dim(A_0)$. Is $J$ necessarily split?
\end{Ques}
Furthermore, we do not fully understand how the role of $I$ relates to the role of the Jacobson radical of prescribed co- and tangent dimension in \cite{BellZhang, Gaddis}. 

A related but distinct question arises when focusing on the subalgebra $A_0$ that is complemented by $I$, rather than the ideal $I$. Recall that $\varphi_0$ was used to produce the isomorphism $A_0 \simeq B_0$, and that we used that for an idempotent $e \in A$, we have that $\varphi_0(e)$ is again idempotent. Denote by $\{f_1, \ldots, f_m\}$ a complete set of primitive idempotents for $B_0$ and by $e_i = \varphi^{-1}(f_i)$ their preimages in $A$. \Cref{Thm: Main thm} then shows that we can automorphically map the set $\{ e_1, \ldots, e_m\}$, which consists of inhomogeneous elements, to their degree $0$ parts $\{ (e_1)_0, \ldots, (e_m)_0\} $ in $A_0$. For the purposes of studying graded projectives and graded simples, it is natural to ask whether one needed to start with $\{e_1, \ldots, e_m\} $ being homogeneous for some grading, and whether one can relax some assumptions on the grading. Note that a version of this question was asked by the author on mathoverflow \cite{Overflow}. 

\begin{Ques}
    Let $A_\bullet$ be a locally finite standard graded semiconnected algebra, and $\{ e_1, \cdots, e_m \}$ a complete set of primitive, pairwise orthogonal idempotents in $A$. Does there exist an automorphism that takes $\{ e_1, \cdots, e_m \}$ to $\{ (e_1)_0, \cdots, (e_m)_0 \}$? If yes, is the same true for $A_\bullet$ being nonnegatively graded and locally finite? 
\end{Ques}

In the spirit of relaxing assumptions, let us discuss what happens when we try to remove or weaken any of the assumptions from \Cref{Thm: Main thm}. We first consider local finiteness and the standardness of the grading. 

\begin{Rem}
    \begin{enumerate}
        \item If we drop the assumption that $A_\bullet$ and $B_\bullet$ are standard graded, then \Cref{Thm: Main thm} fails for trivial reasons. Take e.g. $A_\bullet = k[x]$ the polynomial ring with $x$ in degree $1$, and $B_\bullet = k[y]$ the polynomial ring with $y$ in degree $2$. Clearly $A_\bullet$ and $B_\bullet$ can not be graded isomorphic. However, after scaling down the grading on $B_\bullet$, we do get a graded isomorphism. This suggests that to even generalise the statement meaningfully, we need to relax the notion of graded isomorphism to allow for such kinds of scaling of the grading. 
        \item If we drop the assumption that $A_\bullet$ and $B_\bullet$ are locally finite, every step of the proof of \Cref{Thm: Main thm} fails. We used in each step that our constructed morphism $A_n \to B_n$ is surjective (except for $n=0$, where we got injectivity), and then compared dimensions to conclude that it needs to be bijective. While the arguments still show equality of cardinals $\dim(A_n) = \dim(B_n)$ (assuming an isomorphism in lower degrees and using the axiom of choice), this does not imply that $\psi_n$ needs to be bijective. However, we are not aware of a counterexample in this case either. 
    \end{enumerate}
\end{Rem}

We believe the most interesting direction arises when dropping the semiconnectedness, but keeping the other two assumptions. Then an isomorphism $A \simeq B$ does no longer imply the existence of a graded isomorphism. 

\begin{Exp}
    Let $\Pi(Q)$ be the classical preprojective algebra of an extended Dynkin quiver. It is well-known that $\Pi(Q)$ equipped with the path-length grading is Koszul, see for example \cite{GrantIyama}, and hence locally finite and even semiconnected. However, for the preprojective grading, $\Pi(Q)$ is still generated in degrees $0$ and $1$, and locally finite, but not semiconnected, and hence the two graded algebras can not be graded isomorphic. 
\end{Exp}

The above example, and many similar computations, suggest the following. Since $A_\bullet$ and $B_\bullet$ are locally finite and standard graded, the algebras $A/A_+^2$ and $B/B_+^2$ are finite dimensional. They are in general not isomorphic, but they retain some information about the generators for $A$ and $B$, which in many cases can be seen as ``quivers for $A$ and $B$''.  We therefore raise the following question. 

\begin{Ques}
    Let $A_\bullet$ and $B_\bullet$ be locally finite and standard graded, over an algebraically closed field $k$. If $A \simeq B$ as ungraded algebras, do the finite dimensional algebras $A/A_+^2$ and $B/B_+^2$ have isomorphic Gabriel quivers?
\end{Ques}

If the answer to the above question is positive, one may hope for something stronger: 

\begin{Ques}\label{Ques: Bigrading up to aut?}
    Let $A_\bullet$ and $B_\bullet$ be locally finite and standard graded. If $A \simeq B$, does there exist an isomorphism $\psi \colon A \to B$ such that the induced grading $B = \bigoplus_{i \geq 0} \psi(A_i)$ and $B_\bullet$ form a $\mathbb{Z}^2$-grading?
\end{Ques}

Let us point out that removing both semiconnectedness and local finiteness simultaneously is not possible, even in the commutative case. We present here an example we learned from Isac Hedén which shows not just the failure of \Cref{Thm: Main thm} but an even stronger incompatibility akin to \Cref{Ques: Bigrading up to aut?}. We phrase as much as possible in terms of graded algebras, but encourage the geometrically minded reader to think of $\mathbb{G}_m$-actions and their fixed points on affine varieties instead.  

\begin{Exp}\label{Exp: Isacs example}
    For this example the ground field is $k = \mathbb{C}$. By \cite{Dubouloz}, there exist irreducible affine surfaces $X$ and $Y$ with $\Aut(X)$ finite and such that $X \not \simeq Y$ but $X \times \mathbb{A}^2 \simeq Y \times \mathbb{A}^2$. That means, we have $\mathbb{C}$-algebras $R$ and $S$ such that $R \not \simeq S$ but $R[s,t] \simeq S[u,v]$. Grade the polynomial rings by putting $s$ and $t$ in degree $1$ and similarly $u$ and $v$ in degree $1$. We denote these graded algebras now as $A_\bullet = R[s,t]$ and $B_\bullet = S[u,v] $. Then the degree $0$ parts are $A_0 = R$ and $B_0 = S$ respectively, and hence no graded isomorphism can exist. 
    In fact, these gradings, no matter which isomorphism we use to transport them to the same ring, do not form a $\mathbb{Z}^2$-grading. For a contradiction, suppose that they did. More precisely, fix an arbitrary ungraded isomorphism $\varphi \colon A \to B$, and transport the grading $B_{\bullet} = S[u,v,]$ back to $A_\bullet$, and write $A^{\bullet} = \bigoplus_{i \geq 0} \varphi^{-1}(B_i)$ for this grading. Clearly, we have $A_0 = R = \mathbb{C}[X]$ the coordinate ring of $X$. If $A^\bullet$ and $A_\bullet$ would form a $\mathbb{Z}^2$-grading, we would get a grading on $R = A_0$ as $R = \bigoplus_{i \geq 0} (A_0 \cap A^i)$. However, this is nothing but a $\mathbb{C}^\ast$-action on $\spec(R) = X$. Since $\Aut(X)$ is finite, such an action is trivial, and hence the grading is trivial, meaning that $R \subseteq R \cap A^0 \subseteq A^0$. Furthermore, $A^0 = \varphi^{-1}(B_0) \simeq S$ is isomorphic to the coordinate ring of the surface $Y = \spec(S)$. This means that $X$ contains an irreducible surface isomorphic to $Y$. Since $Y$ is also irreducible, it follows that $X \simeq Y$, contradicting the assumption. 
\end{Exp}

We therefore conclude with the following reformulation of \Cref{Ques: Bigrading up to aut?} in the commutative case, as it may also be of interest to the study of automorphism groups of affine varieties. For this topic, we refer the reader to \cite{KraftAutomorphisms}.  

\begin{Ques}
    Let $X$ and $Y$ be affine varieties over $\mathbb{C}$, equipped with $\mathbb{C}^\ast$-actions $\alpha \colon \mathbb{C}^\ast \to \Aut(X) $ and $\beta \colon \mathbb{C}^\ast \to \Aut(Y) $ with strictly positive weights and finitely many fixed points. If $X \simeq Y$, does there exist an isomorphism $\varphi \colon X \to Y$ such that $\Im(\alpha)$ and $\varphi^{-1}(\Im(\beta))$ are conjugate in $\Aut(X)$?
\end{Ques}

\section*{Acknowledgement}
We thank Jason Gaddis for helpful correspondence. We thank Isac Hedén for discussions on automorphisms of affine varieties, and for explaining \Cref{Exp: Isacs example} to us.

\printbibliography

\end{document}